\documentclass[twoside]{article}

\usepackage[utf8]{inputenc}
\usepackage[english]{babel}
\usepackage{amsthm}
\usepackage{graphicx}

\theoremstyle{plain}
\newtheorem{theorem}{Theorem}
\newtheorem{lemma}{Lemma}
\newtheorem{corollary}{Corollary}

\newtheorem{assertion}{Assertion}

\theoremstyle{definition}
\newtheorem{definition}{Definition}
\newtheorem{remark}{Remark}

\begin{document}

\title{On the matching arrangement of a graph and properties of its characteristic polynomial}
\author{A.I.Bolotnikov\footnote{Job: Moscow State University, e-mail: bolotnikov-94@mail.ru}
}

\maketitle

\begin{abstract} 
This paper considers a hyperplane arrangement constructed with a subset of a set of all simple paths in a graph. A connection of the constructed arrangement to the maximum matching problem is established. Moreover, the problem of finding the characteristic polynomial is reduced to the case of a connected initial graph. The formula of the characteristic polynomial was also found for the case, when the initial graph is a tree.
\end{abstract}


\section{Introduction} %   Раздел

The number of regions of a hyperplane arrangement can be calculated with its characteristic polynomial \cite{zasl}. This result is useful in solving problems of enumerative combinatorics. For example, in \cite{irm}, \cite{irm2}, \cite{irm3}  a special hyperplane arrangement with several related geometric constructions was introduced in order to estimate the number of threshold functions.

In the paper \cite{gre-zasl} the connection between the chromatic polynomial of a graph and the number of acyclic orientations of that graph is depicted as a property of a specific hyperplane arrangement. This arrangement is called a graphical arrangement, and the set of all edges of the initial graph is used for its construction. The characteristic polynomial of the graphical arrangement is equal to the chromatic polynomial of the initial graph, and the number of regions of the graphical arrangement is equal to the number of acyclic orientations in the initial graph.

The first part of this paper contains basic theory on hyperplane arrangements and partially ordered sets. Papers \cite{stan}, \cite{stan2},\cite{annal} include an extensive number of results on hyperplane arrangements and partially ordered sets, and in particular they contain definitions and assertions from the first part.

 In the second part a subset of a set of all simple paths in a graph is used to construct a new hyperplane arrangement called the matching arrangement. It is shown that the matching arrangement is connected to a maximum matching problem, and the possibility that two non-isomorphic graphs have the same matching arrangements is also been explored.

The third part contains several results on the characteristic polynomial of the matching arrangement. The first result is that the characteristic polynomial of the matching arrangement is a graph invariant. The second result is that the characteristic polynomial of an unconnected graph is equal to a product of characteristic polynomials of that graph's connected components. The third result describes specific sets of non-isomorphic graphs with equal characteristic polynomials of matching arrangements. Specifically, for a case when the initial graph is a tree, a formula for the characteristic polynomial is obtained.

\section{Basic theory on hyperplane arrangements}

\begin{definition} \textit{A hyperplane arrangement} is a finite set of affine hyperplanes in some
vector space V. In this paper $V = R^n$.
\end{definition}
\begin{definition}If all hyperplanes intersect in one point, an arrangement is called \textit{central} .
\end{definition}
\begin{definition}
\textit{A region} of an arrangement  is a connected component of
the complement  of the union of hyperplanes.
\end{definition}
\begin{definition}
\textit{A rank} of an arrangement is the dimension
of the space spanned by the normal vectors to the hyperplanes in that arrangement.
\end{definition}
\begin{definition}
For a hyperplane arrangement A a partially ordered set (or poset) $L(A)$ is defined the following way:  elements of L(A) are all
nonempty intersections of hyperplanes in A, and $x \leq y$ in L(A), if $x \supseteq y $.
\end{definition}
Each element x of the set L(A) is an affine subspace V. Its dimension will be denoted as dim(x).

\begin{definition}
\textit{M\"{o}bius function} on a partially ordered set P is defined by following conditions:
\begin{enumerate}
    \item $\mu (x,x) = 1$ $ \forall x \in P$
    \item $\mu (x,y) = -\sum_{x \leq z < y} \mu (x,z)$ $ \forall x<y \in P $
    \item $\mu (x,y) = 0$, if x and y are not comparable
\end{enumerate}.
\end{definition}
If P have a minimal element $\hat 0$, then $\mu (x)$ is a denotion for $\mu (\hat{0}, x)$.
\begin{definition}
\textit{A direct product} of posets P and Q is defined the following way:

$ P\times Q = \{(s,t):s\in P, t\in Q\}$, where $(s,t)\leq (\dot s,\dot t)\Leftrightarrow s\leq \dot s $ and $t \leq \dot t$
\end{definition}

\begin{assertion}
$\mu_{P \times Q}((s,t),(\dot s,\dot t)) =\mu_{P}(s,\dot s)\cdot \mu_{Q}(t,\dot t) $
\end{assertion}
\begin{definition}
Two posets are \textit{isomorphic}, if there is an order-preserving bijection between them whose inverse is also order-preserving.
\end{definition}
\begin{definition}
\textit{A characteristic polynomial} of a hyperplane arrangement A is defined by the following formula:
\[ \chi_{A} (t) = \sum_{x \in L(A)} \mu(x)t^{dim(x)}\].
\end{definition}

\begin{definition}
\textit{A matroid} is a pair $M = (S,J)$, where S is a finite set and
J is a collection of subsets of S, satisfying the following axioms:
\begin{enumerate}
    \item J is not empty, and if $K \in J, N \subset K$, then $N \in J$
    \item $\forall T \subset S $ the maximal elements of $J \cap 2^{T}$ have the same cardinality.
\end{enumerate}
\end{definition}
\begin{definition}
Two matroids $M = (S,J)$ and $\hat M = (\hat S, \hat J)$ are isomorphic, if there is a bijection $f:S \rightarrow \hat S$ , such that $\{x_1,x_2,...,x_k\}\in J \Leftrightarrow \{f(x_1),f(x_2),...f(x_k)\} \in \hat J$
\end{definition}
\begin{definition}
Let T be a subset of S. \textit{The rank} of T is defined by: $rk(T) = max(|I|: I \in J, I \subseteq S) $. The rank of a matroid M(S,J) is the rank of S.
\end{definition}
\begin{definition}
 \textit{A k-flat} of a matroid is a maximal subset of rank k.
\end{definition}
\begin{definition}
For a matroid M define L(M) to be the poset of flats of M, ordered by inclusion. 
The characteristic polynomial of a matroid m is defined by
\[ \chi_{M} (t) = \sum_{x \in L(M)} \mu(x)t^{r-rk(x)}\]

, where r is a rank of M.
\end{definition}

\begin{definition}
A matroid M is simple, if it does not contain elements x, such that $rk(\{x\}) =0$ (loops) and it does not contain elements x,y, such that $rk(\{x,y\}) =1 $ (parallel points).
\end{definition}

\begin{assertion}Let $M_1$ and $M_2$ be simple matroids. Matroids $M_1$ and $M_2$ are isomorphic, if and only if posets $L(M_1)$ and $L(M_2)$ are isomorphic.

\end{assertion}
\begin{definition}
For a hyperplane arrangement A the matroid $M_A(S,J)$ can be defined the following way:  S is a set of normal vectors to hyperplanes from A, with each hyperplane having exactly one normal vector in S, and elements of J are sets of linearly independent vectors from S. 
\end{definition}
\begin{assertion}
 If A is central, then $M_A(S,J)$ is simple, and $L(M) \cong L(A)$.
As a corollary, $\chi_{A} (t) = t^{n-r(A)}\cdot \chi_{M_A} (t)$, where r(A) is the rank of A, and n is the dimension of the linear space that contains A.
\end{assertion}
\begin{definition}
\textit{A proper coloring} of a graph G(V,E) on [n] is a map $C:V\rightarrow [n]$, such that if vertices $v_1$ and $v_2$ are adjacent, then $c(e_1) \neq c(e_2)$.
\end{definition}
\begin{definition}
\textit{The chromatic polynomial} $\chi_G$ is a polynomial with the following property: for each natural n $\chi_{G}(n)$ is equal to the number of proper colorings on [n]. 
\end{definition}
\begin{definition}
\textit{An orientation} of a graph is an assignment of a direction to each edge of that graph. An orientation is called acyclic, if it does not contain any directed cycles.
\end{definition}
\begin{definition}
  \textit{A graphical arrangement} of a graph G(V,E) with |V| = n in space $R^n$ is a hyperplane arrangement with a hyperplane $x_i - x_j = 0$ for every edge $(v_i,v_j)$ .
 \end{definition}
 \begin{assertion}
 A graphical arrangement has the following properties:
 \begin{enumerate}
     \item Its characteristic polynomial is equal to the chromatic polynomial of the initial graph. 
     \item Every acyclic orientation of a graph bijectively corresponds to a region of the graphical arrangement.
  \end{enumerate}
\end{assertion}

\section{Definition of an arrangement MA(G,N) and its properties}

\begin{definition}
Let G(V,E), |E|=n be a graph without loops and parallel edges, let $N:E\rightarrow \{1,2,..,n\}$ be a numeration of edges in G. Let P be a sequence of edges $(r_1,r_2,r_3,...,r_n)$ that form a simple path or a simple cycle with even number of edges in graph G. A hyperplane that corresponds to sequence P is a hyperplane with an equation $x_1-x_2+x_3 -...x_n=0$. Let F be a set of all sequences of edges in G that form a simple path or a simple cycle with even number of edges in G. Then matching arrangement is a set of all hyperplanes that correspond to elements of F. A matching arrangement will be denoted as MA(G,N).           
\end{definition}
\begin{definition}
 Every hyperplane arrangement corresponds to the set of all equations of its hyperplanes. If there are two hyperplane arrangements with identical sets of equations, then such arrangements will be call \textit{identical}.
\end{definition}

\begin{theorem}
Let $G_1$, $G_2$ be two connected graphs without loops, parallel edges and isolated vertices, and there are a numeration of edges $N_1$ in $G_1$ and a numeration of edges $N_2$ in $G_2$, such that $MA(G_1,N_1)$ and $MA(G_2,N_2)$ are identical. Then $G_1$ and $G_2$ are isomorphic, except the case when one of the graphs is isomorphic to $K_3$, and another is isomorphic to $K_{1,3}$.
\end{theorem}

\begin{proof}[Proof]
 An arrangement MA(G,N) contains all hyperplanes that correspond to sequences of length 2. An equation of such hyperplane is $x_i - x_j = 0$, and if there is a hyperplane with this equation in MA(G,N),then $e_i$ and $e_j$ are adjacent in G (and visa versa, if there is no corresponding hyperplane in MA(G,N) for $e_i$ and $e_j$, then these edges are not adjacent). Therefore, an edge graph of an initial graph G can be reconstructed from MA(G,N) the following way. The number of vertices of the edge graph can be deduced from the dimension of a vector space that contains MA(G,N). Let $\{v_1,..,v_n\}$ be a set of vertices of the edge graph of G. Then $v_i$ and $v_j$ are connected if and only if a hyperplane $x_i - x_j = 0$ is an element of MA(G,N).

 Therefore, if $MA(G_1,N_1)$ and $MA(G_2,N_2)$ are identical, then edge graphs of $G_1$ and $G_2$ are isomorphic. According to Whitney theorem\cite{witney}, if two graphs have edge graphs that  are isomorphic, then initial graphs are isomorphic as well. The only exception is when one graph is isomorphic to $K_{3}$, and another is isomorphic to $K_{1,3}$. 
 
 If $G_1\cong K_{3}$ and $G_2 \cong K_{1,3}$, then the sets of hyperplane equations for both $MA(G_1,N_1)$ and $MA(G_2,N_2)$ consist of equations $x_1 = 0$, $x_2 = 0$, $x_3 = 0$, $x_1 - x_2 = 0$, $x_2 - x_3 = 0$, $x_3 - x_1 = 0$, therefore $MA(G_1,N_1)$ and $MA(G_2,N_2)$ are identical.
\end{proof}

\begin{remark}
The result on theorem 1 can be expanded to cover the case of non-connected graphs the following way. Any two edges  $e_i$ and $e_j$ of a graph G are in the same connected component, if and only if there is a  hyperplane in MA(G,N), whose equation contains both $x_i$ and $x_j$. Let $G_1$ and $G_2$ be two non-connected graphs with identical arrangements $MA(G_1,N_1)$ and $MA(G_2,N_2)$. Then equations of hyperplanes of $MA(G_1,N_1)$ and $MA(G_2,N_2)$ can be divided into groups that correspond to connected components of the initial graphs, and then theorem 1 can be applied two each group independently. Let $G_1$ contain p connected components that are isomorphic to $K_3$, and q components that are isomorphic to $K_{1,3}$. Let $G_2$ contain s and r components of the same types respectively. Let $p+q=s+r$, $p\neq s$, $q\neq r$, and for every other component A of $G_1$ there is one and only one component B in $G_2$, such that A isomorphic to B. Then there are numerations $N_1$ and $N_2$, such that $MA(G_1,N_1)$ and $MA(G_2,N_2)$ are identical, but $G_1$ is not isomorphic to $G_2$. However, if graphs $G_1$ and $G_2$ don't contain components that are isomorphic to $K_3$ or  $K_{1,3}$, and $MA(G_1,N_1)$ and $MA(G_2,N_2)$ are identical,then $G_1$ and $G_2$ are isomorphic.

\end{remark}

\begin{theorem}
Let G(V,E), |E|=n be a graph without loops and parallel edges. Let D be a region of an arrangement $MA(G,N)$, and let $a = (a_1,..,a_n)$ and $b=(b_1,...b_n)$ be two vectors from D. Let $\Pi_1$ be a matching with the greatest sum of weights of its edges if a weight of an edge $e_i$ is equal to $a_i$. Let $\Pi_2$ be a matching with the greatest sum of weights of its edges if a weight of an edge $e_i$ is equal to $b_i$.  Then $\Pi_1 = \Pi_2$.

\end{theorem}

\begin{proof}[Proof]
By contradiction.Assume that for two vectors a,b from the same region D of MA(G,N) the corresponding maximum matchings are different. Let $H = \Pi_1 \triangle \Pi_2$ be a symmetric difference of $\Pi_1$ and $\Pi_2$.  If there is a vertex with degree more then 2 in H, then there are two adjacent edges from the same matching, which contradicts the definition of a matching. Therefore, all vertices in H have degree less or equal to 2, and each of the connected components of H is either a simple path or a simple cycle. Edges from the same matching can't be adjacent, so edges belong  to $\Pi_1$ and to $\Pi_2$ in rotation in each component of H. Therefore, a component of H can't be a simple cycle with odd number of edges. Let C be one of the components of H and let $x_{i_1} - x_{i_2} + ... x_{i_m}$ be the corresponding expression. Assume that values of this expression on a and b have the same sign. Without loss of generality, let's assume, that these values are negative. Without loss of generality, let $\Pi_1$ be the matching that contains edges $x_{i_1},x_{i_3},...$, and let a be a vector of weights, for which $\Pi_1$ is supposed to be the maximum matching. Then the matching $\Pi_3$, obtained  by replacing edges $x_{i_1},x_{i_3},...$ with edges $x_{i_2},x_{i_4},...$  in $\Pi_1$,has greater total weight, then $\Pi_1$, with a as a vector of weights. This contradicts the initial assumption that $\Pi_1$ is a maximum matching.
Therefore, the value of the expression $x_{i_1} - x_{i_2} + ... x_{i_m}$ is positive in one of the vectors a and b, an negative in another. This means that a hyperplane with the equation  $x_{i_1} - x_{i_2} + ... x_{i_m} = 0$ divides a and b, therefore a and b belong to different regions of the arrangement. This contradicts the condition that a and b both belong to D.
\end{proof}

\section{Characteristic polynomial of the arrangement MA(G,N)}
\begin{theorem}(\textit{On characteristic polynomial as a graph invariant}): Let $MA(G,N_1)$ and $MA(G,N_2)$ be two hyperplane arrangements constructed from two different numerations of a graph G. Then $\chi_{MA(G,N_1)}(t) = \chi_{MA(G,N_2)}(t)$.
\end{theorem}

\begin{proof}[Proof]
For the arrangement $MA(G,N_1)$ elements of its matroid will be chosen the following way. If a hyperplane's equation is $x_i = 0$, then $e_i$ will be chosen as a normal vector to this hyperplane, where $\{e_1,e_2,...,e_n\}$ is a standard basis. For every other hyperplane with an equation $x_{i_1} - x_{i_2} + x_{i_3} - x_{i_4}... = 0$ one of two vectors that can be written as $e_{i_1} - e_{i_2} + e_{i_3} - e_{i_4} +...$ will be chosen arbitrarily. For the arrangement $MA(G,N_2)$  elements of its matroid will be chosen the following way. If a hyperplane's equation is $x^i = 0$, then $e^i$ will be chosen as a normal vector to this hyperplane, where $\{e^1,e^2,...,e^n\}$ -is a standard basis.  Otherwise let a hyperplane H in $MA(G,N_2)$ with an equation $x^{i_1} - x^{i_2} + x^{i_3} - x^{i_4}... = 0$ correspond to a simple path $(r_1,r_2,r_3,...r_k), r_i \in E$ in G, and let $H'$ in $MA(G,N_1)$ correspond to the same path.If vector $e_{N_1(r_{i_1})} - e_{N_1(r_{i_2})} + e_{N_1(r_{i_3})} - e_{N_1(r_{i_4})} +...$ was chosen as a normal vector to H', then vector $e^{N_2(r_{i_1})} - e^{N_2(r_{i_2})} + e^{N_2(r_{i_3})} - e^{N_2(r_{i_4})} +...$ will be chosen as a normal vector to H.

 Let  $M_1(S_1,J_1)$ be a matroid of the arrangement $MA(G,N_1)$ and $M_2(S_2,J_2)$ be a matroid of the arrangement $MA(G,N_2)$. $MA(G,N_1)$ and $MA(G,N_2)$ are central, therefore posets of $MA(G,N_1)$ and $MA(G,N_2)$  are isomorphic to posets of $M_1$ and $M_2$ respectively.

 Let $U:R^n\rightarrow R^n$ be a linear operator, such that $U(e_i) = e^j$ if and only if $N_1(x) =i$, $N_2(x) = j$ for some edge x in G. Each element $y\in S_1$ can be expressed as $y = e_{i_1} - e_{i_2}+...e_{i_k}$, so $U(y) = e^{j_1} -e^{j_2} +...e^{j_k}$. Both y and U(y) correspond to the same edge sequence in G, therefore U defines a bijection between $S_1$ and $S_2$. Since U is an invertible linear operator, this bijection between $S_1$ and $S_2$ is and 
 isomorphism between $M_1$ and $M_2$.

Since matroids $M_1$ and $M_2$ are isomorphic, posets $L(M_1)$ and $L(M_2)$ are isomorphic. Since these posets are isomorphic, posets of the initial arrangements are also isomorphic.
 The isomorphism of posets $L(MA(G,N_1))$ and $L(MA(G,N_2))$ will be denoted as V.
 According to definition, $\chi_{MA(G,N_1)} (t) = \sum_{x \in L(MA(G,N_1))} \mu_{L(MA(G,N_1))}(x)t^{dim(x)}$. Since $L(MA(G,N_1))$ and $L(MA(G,N_2))$ are isomorphic, $\chi_{MA(G,N_2)} (t) = \sum_{x \in L(MA(G,N_1))} \mu_{L(MA(G,N_2))}(V(x))t^{dim(V(x))}$.  $\mu(x) = \mu(V(x))$, because V is an isomorphism.   $dim(x)= dim(V(x))$, because $L(MA(G,N_1))$ and $L(MA(G,N_2))$ are isomorphic and both arrangements belong to linear spaces of the same dimension. Therefore, every member of the sum in $\chi_{MA(G,N_1)} (t)$ is equal to its corresponding member of the sum in $\chi_{MA(G,N_2)} (t)$, so $\chi_{MA(G,N_1)}(t) =\chi_{MA(G,N_2)}(t)$.

\end{proof}

\begin{theorem}(\textit{On a characteristic polynomial for a non-connected graph}): Let $G(V,E)$ have the following property: there are set $V_1,V_2$, such that $V  = V_1 \cup V_2$,  $V_1 \cap V_2 = \emptyset$ , and there is no edge $e = (v_1,v_2)$ that connects a vertex from $V_1$ and a vertex from $V_2$. Let $G_1(V_1, E_1),G_2(V_2,E_2)$ be subgraphs of G, $E = E_1 \cup E_2$. Then $\chi_{MA(G,N)}(t) = \chi_{MA(G_1,N_1)}(t) \cdot \chi_{MA(G_2,N_2)}(t)$.
\end{theorem}

\begin{proof}[Proof]
Let $|E_1| = k,|E_2| = m, |E| = k+m$. The arrangements MA(G,N),$MA(G_1,N_1)$, $MA(G_2,N_2)$ belong to a linear spaces $R^{k+m}$, $R^k$,  $R^m$ respectively. Let $e_1,e_2,...,e_k$ be a basis of $R^k$,  let $e^1,e^2,....,e^m$ be a basis $R^m$. Each vector $e_i$ corresponds to an edge in $G_1$, and each vector $e^i$ corresponds to an edge in $G_2$. Each vector of the basis of $R^{k+m}$ corresponds to an edge in G, which means it corresponds either to a edge in $G_1$ or to an edge in $G_2$. In the first case this vector will be denoted as $f(e_i)$, where $e_i$ is a vector that corresponds to the same edge, and in the second case it will be denoted as $g(e^i)$, where $e^i$ is a vector that corresponds to the same edge. Therefore, the basis in $R^{n+m}$ will be denoted as $\{f(e_1),f(e_2),...,f(e_k),g(e^1),g(e^2),...,g(e^m)\}$.

Let  $M(S,J),M_1(S_1,J_1),M_2(S_2,J_2)$ be matroids of the arrangements $MA(G,N),MA(G_1,N_1),MA(G_2,N_2)$. Elements of these matroids are chosen the following way. For an arrangement with an equation $x_i = 0$ from $MA(G_1,N_1)$ vector $e_i$ is chosen as a normal vector. For every other hyperplane $MA(G_1,N_1)$ one of two vectors that can be written as  $e_{i_1} - e_{i_2} +...$ is chosen arbitrarily.This is how all elements of $S_1$ are chosen.  Elements of $S_2$ are chosen in a similar way. Let H be an arrangement in $MA(G,N)$ that corresponds to a simple path $(r_1,..,r_k), r_i \in E$ in G, such that all $r_i$ belong to $G_1$, and let H' be a hyperplane from $MA(G_1,N_1)$ that corresponds to the same path. Then if vector $e_{N_{1}(r_1)} - e_{N_{1}(r_2)} + e_{N_{1}(r_3)} ....$ was chosen for H', then vector $f(e_{N_{1}(r_1)}) - f(e_{N_{1}(r_2)}) + f(e_{N_{1}(r_3)}) ....$ is chosen for H.
 
 The partially ordered sets of arrangements  $MA(G,N),MA(G_1,N_1),MA(G_2,N_2)$ are isomorphic to the partially ordered sets of matroids $M(S,J),M_1(S_1,J_1),M_2(S_2,J_2)$ respectively.
Vectors $e_1,e_2,...,e_k$ belong to $S_1$, vectors $e^1,e^2,....,e^m$ belong to $S_2$, vectors $\{f(e_1),f(e_2),...,f(e_k),g(e^1),g(e^2),...,g(e^m)\}$ belong to S, which means that $rk(M_1) = k,rk(M_2) = m, rk(M) = k+m$. Each element of $S_1$ is normal to a hyperplane from $MA(G_1,N_1)$, it corresponds to a simple path in $G_1$ and is equal to $e_{i_1} - e_{i_2} + e_{i_3} -...$ for a certain set of vectors $e_{i_1},e_{i_2},e_{i_3},...$. In the same way each element of $S_2$ is normal to a hyperplane from $MA(G_2,N_2)$, it corresponds to a simple path in $G_2$ and is equal to $e^{i_1} - e^{i_2} + e^{i_3} -...$ for a certain set of vectors $e^{i_1},e^{i_2},e^{i_3}...$.Each element of S is normal to a hyperplane from $MA(G,N)$, it corresponds to a simple path $G$, which means it corresponds to a simple path in either $G_1$ or $G_2$. Then it is equal to either $f(e_{i_1}) - f(e_{i_2}) + f(e_{i_3}) -...$ or $g(e^{i_1}) - g(e^{i_2}) + g(e^{i_3}) -...$ and it also corresponds to an element of either $S_1$ or $S_2$.The last two correspondences will now be defined more accurately.  Let x = $\alpha_{1}e_1 + \alpha_{2}e_2 +...+ \alpha_{k}e_k$ be an element of $S_1$. Then a corresponding element in $R^{k+m}$ is $\alpha_{1}f(e_1) + \alpha_{2}(e_2) +...+ \alpha_{k}f(e_k)$. Correspondence for elements of $S_2$ is defined in a similar way.

The next step is to prove that the partially ordered set P of the matroid M is isomorphic to a direct product of partially ordered sets $P_1,P_2$ of matroids $M_1$ and $M_2$. It will be proven by construction of an isomorphism from $P_1 \times P_2$ to P. Let $B = \{x_1,x_2,...,x_p\}$ be an n-facet in $M_1$, and let $C = \{x^1,x^2,...,x^q\}$ be an l-facet in $M_2$. Then an element $(B,C)\in P_1 \times P_2$ corresponds to a set $O =\{f(x_1),f(x_2),...,f(x_p),g(x^1),g(x^2),...,g(x^q)\}$. Without loss of generality, let $\{x_1,x_2,...,x_n\}$ be a set of n linearly independent vectors from B, and let $\{x^1,x^2,...,x^l\}$ be a set of l independent vectors from C.  Let $\sum_{i} \alpha_i \cdot f(x_i) + \sum_{j}\beta_j \cdot g(x^j)$ be a non-trivial linear combination of vectors $\{f(x_1),f(x_2),...,f(x_n)\}$ and $\{g(x^1),g(x^2),...,g(x^l)\}$.  $\\sum_{i} \alpha_i \cdot f(x_i)$ can be written as a linear combination of $f(e_i)$, and $\sum_{j}\beta_j \cdot g(x^j)$ can be written as a linear combination of $g(e^j)$. Since $\{f(e_1),f(e_2),...,f(e_k),g(e^1),g(e^2),...,g(e^m)\}$ is a basis of space $R^{k+m}$ the sum of these two combinations is equal to 0, if and only if each of these two combinations is equal to 0. And since $\{f(x_1),f(x_2),...,f(x_n)\}$  and $\{g(x^1),g(x^2),...,g(x^l)\}$ are two sets of linear independent vectors, then each combination is equal to 0, if and only if all coefficients are equal to 0. Therefore, the union of sets  $\{f(x_1),f(x_2),...,f(x_n)\}$ and $\{g(x^1),g(x^2),...,g(x^l)\}$ is a set of n+l linearly independent vectors. Lastly, every element of O can be written as a linear combination of these vectors, so the rank of the set O is equal to n+l.

Now suppose that S contains an element x that can be written as a linear combination of these (n+l) vectors, but it does not belong to O. Element x can be a linear combination of either only vectors $f(e_i)$, or only vectors $g(e^j)$,so it can be a linear combination of either only $\{f(x_1),f(x_2),...,f(x_n)\}$, or only $\{g(x^1),g(x^2),...,g(x^l)\}$. Therefore, there is an element in either $S_1$ or $S_2$, that can be written as a linear combination of $\{x_1,x_2,...,x_n\}$ (or $\{x^1,x^2,...,x^l\}$, respectively) but does not belong to B (or C, respectively). This means that either B is not a facet of $M_1$,or C is not a facet of$M_2$, which is a contradiction. Therefore, O is a (n+l)-facet of M.

The constructed function will be denoted as F. For example, $O = F((B,C))$. F matches elements of $P_1 \times P_2$ with elements of $P$.  F is injective, because  f and g are injective, and because elements of any set from P can be divided in a unique way into elements like $f(x_i)$ and elements like $g(x^j)$. The next step is to prove, that F is surjective. Let O be an element of P,which means that O is an n-facet in  and a set of vectors $\{f(x_1),f(x_2),...,f(x_p),g(x^1),g(x^2),...,g(x^q)\}$. Let B be $\{x_1,x_2,...,x_p\}$, and let C be $\{x^1,x^2,...,x^q\}$. If B is not a facet in $M_1$, then there is an element x in $S_1$, that can be written as a linear combination of $x_1,x_2,...,x_p$, but does not belong to B. Therefore, the element f(x) can be written as a linear combination of $\{f(x_1),f(x_2),...,f(x_p)\}$, but does not belong to O. This means that O is not a facet of M, which is contradiction. In turn, this means that B is a facet of $M_1$. In the same way C is a facet of $M_2$. By definition of F, $F((B,C)) = O)$. This concludes the proof that F is surjective. 

As a result,  F is a bijection between $P_1\times P_2$ and P. $(B_1,C_1)\leq (B_2,C_2) $, if and only if each element of $B_1$ is an element of $B_2$,and each element of $C_1$ is an element of $C_2$. On the other hand,  $O_1 = F((B_1,C_1))\leq F((B_2,C_2))= O_2 $ ,if and only if each element of $O_1$ is an element of $O_2$. Each element of $O_2$ can be written either  as $f(x),x\in B_2$, or as $g(y),y\in C_2$,and each element of $O_1$ can be written either as $f(x),x\in B_1$, or as $g(y),y\in C_1$
All elements of type $f(x),x\in B_1$ of the set $O_1$ belong to $O_2$, if and only if all elements of $B_1$ belong to $B_2$, and all elements of type $g(y),y\in C_1$  of the set $O_1$ belong $O_2$, if and only if all elements of $C_1$ belong to $C_2$. Therefore,  $(B_1,C_1)\leq (B_2,C_2) $ is equivalent to $F((B_1,C_1))\leq F((B_2,C_2))$, which means that F preserves order in partially ordered sets. As a result, F is an isomorphism of partially ordered sets.

The characteristic polynomial of M can now be transformed the following way.
\[\chi_{M} (t) = \sum_{x \in L(M)} \mu(x)t^{rk(M)-rk(x)} = \sum_{x=F((B\times C))} \mu(x)t^{rk(M)-rk(x)}\]\[ = \sum_{x=F((B\times C))} \mu(B) \cdot \mu(C) \cdot t^{rk(M_1)+rk(M_2)-rk(B) - rk(C)}\]\[ = \sum_{x=F((B\times C))} \mu(B)\cdot t^{rk(M_1)-rk(B)} \cdot \mu(C)\cdot t^{rk(M_2)- rk(C)}\]\[ =  \sum_{B \in L(M_1)}\mu(B)t^{rk(M_1)-rk(B)} \cdot  \sum_{C \in L(M_2)}\mu(C)t^{rk(M_2)-rk(C)} = \chi_{M_1} (t) \cdot \chi_{M_2} (t)\] 
Since the dimensions of linear spaces that contain normal vectors to arrangements of hyperplanes $MA(G,N),MA(G_1,N_1),MA(G_2,N_2)$ are equal to the dimensions of spaces that contain $MA(G,N),MA(G_1,N_1),MA(G_2,N_2)$, then $\chi_{MA(G,N)} (t) = \chi_{M} (t),\chi_{MA(G_1,N_1)} (t) = \chi_{M_1} (t),\chi_{MA(G_2,N_2)} (t) = \chi_{M_2} (t) $,  therefore $\chi_{MA(G,N)}(t) = \chi_{MA(G_1,N_1)}(t)\cdot\chi_{MA(G_2,N_2)}(t)$. 
\end{proof}
 \begin{corollary}
 Let G be a graph without loops, parallel edges and isolated vertices, and let $G_1, G_2,...,G_k$ be the connected components of G. 

Then $ \chi_{MA(G,N)} (t) =  \chi_{MA(G_1,N_1)}\cdot\chi_{MA(G_2,N_2)}\cdot...\cdot\chi_{MA(G_k)}$

 \end{corollary}

\begin{theorem}
(\textit{On the characteristic polynomial in the tree case}): Let graph G(V,E), |E|=n be a tree. Then the characteristic polynomial of MA(G,N) is equal to $\chi_{MA(G,N)} (t) = (t-1)(t-2)...(t-n) $.

\end{theorem}

\begin{proof}[Proof]
Since G is a tree, it has n+1 vertices. Let $K_{n+1}$ be a complete graph with n+1 vertices and let vertices of $K_{n+1}$ and G have a one-to-one correspondence between each other, specifically let a vertex $v_i$ of G correspond to a vertex $v^i$ of $K_{n+1}$.

Let $M_G (S_G, J_G)$ be a matroid of $MA(G,N)$.Elements of $M_G$ are chosen in s way similar to one in the beginning of theorem 1. The partially ordered set of $MA(G,N)$ is isomorphic to the partially ordered set of $M_G$. Let $e_1,e_2,..e_n$ be the basis of the space $R^n$ that contains $MA(G,N)$, such that each vector $e_i$ corresponds to an edge $r_i$ in G. Each element of $S_G$ is a normal vector to a hyperplane from $MA(G,N)$, and it can be written as $e_{i_1} - e_{i_2} + e_{i_3} ... +e_{i_m}$, where $e_{i_1},e_{i_2},...,e_{i_k}$ is as sequence of vectors, such that $r_{i_1},r_{i_2},...,r_{i_k}$ is a simple path in G.

Let $\{e^1,e^2,...,e^{n+1}\}$ be the standard basis in the space that contains $A_K$. The following step is to construct a linear operator with a following property. This operator will convert elements of $S_G$ into elements of $R^{n+1}$, such that an element corresponding to a simple path in G that connects vertices $v_i$ and $v_j$ will be converted to an a normal vector of a hyperplane that corresponds to an edge in $K_{n+1}$ that is incident to $v^i$ and $v^j$. This property will further be referred to as Property 1. 

Each vector from  $\{e_1,e_2,e_3,...,e_n\}$ corresponds to a simple path of length 1 in G ,i.e. and edge in G. Let all vertices of G be coloured in black and white, such that adjacent vertices are coloured in different colours. Since G is a tree, such colouring is possible. The following is the definition of the function $F:\{e_1,e_2,e_3,...,e_n\} \rightarrow R^{n+1}$. If  $e_i$ corresponds to an edge that is incident to vertices $v_i$ and $v_j$,where $v_i$ is black, and $v_j$ is white , then $F(e_i) = e^i - e^j$.

Let $v_{i_1},v_{i_2},v_{i_3},...v_{i_{t+1}}$ be a sequence of vertices of an arbitrary simple path in G, an let $r_{i_1},r_{i_2},r_{i_3},...r_{i_t}$ be a sequence of edges of the same path.
Let vertex $v_{i_1}$ be a black vertex. The following is the proof by induction that $F(e_{i_1}) -F(e_{i_2}) + F(e_{i_3}) - F(e_{i_4})+... F(e{i_{t+1}})) = e^{i_1} - e^{i_{t+1}}$.

Base case: t=1. $F(e_{i_1}) = e^{i_1} - e^{i_t+1}$ by definition of F.

Inductive step. Let $F(e_{i_1}) -F(e_{i_2}) + F(e_{i_3}) - F(e_{i_4})+...F(e{i_k}) = e^{i_1} - e^{i_{k+1}}$ stand for every simple path of length k in G. To prove that $F(e_{i_1}) -F(e_{i_2}) + F(e_{i_3}) - F(e_{i_4})+...F(e{i_{k+1}}) = e^{i_1} - e^{i_{k+2}}$. stands for every simple path of length k+1, the only two possible cases are considered:

\begin{enumerate}
    \item k+1 is odd and $v_{i_{k+1}}$ is black. Then by induction hypothesis $F(e_{i_1}) -F(e_{i_2}) + F(e_{i_3}) - F(e_{i_4})+...-F(e_{i_k})+F(e_{i_{k+1}}) = e^{i_1} - e^{i_{k+1}} + e^{i_{k+1}} - e^{i_{k+2}} = e^{i_1} - e^{i_{k+2}}  $
    \item k+1 is even and $v_{i_{k+1}}$ is white. Then by induction hypothesis $e_{i_1} -e_{i_2} + e_{i_3} - e_{i_4}+...+e_{i_k}-e_{i_{k+1}} =e^{i_1} - e^{i_{k+1}} - (e^{i_{k+2}} - e^{i_{k+1}}) = e^{i_1} - e^{i_{k+2}}  $.
\end{enumerate}

Inductive step is proved. The similar proof works for the case when $F(e_{i_1})$ is white. This result allows to redefine F as linear operator that converts $R^n$ into the subset of $R^{n+1}$, such that it has the Property 1.

A subgraph of $K_{n+1}$ that consists of edges corresponding to $\{F(e_1),F(e_2),F(e_3),...,F(e_n)\}$ and graph G are isomorphic, because an edge $(v^j,v^k)$ corresponds to $F(e_i)$ if and only if a path that connects $v_j$ and $v_k$ corresponds to $e_i$, and paths that correspond to vectors $e_i$ are edges. Now assume that there is a non-trivial linear combination of vectors $\{F(e_1),F(e_2),F(e_3),...,F(e_n)\}$ that is equal to 0. Let $W\subseteq \{F(e_1),F(e_2),F(e_3),...,F(e_n)\} $ be a set of vectors that have non-zero coefficients in such combination. A subgraph of $K_{n+1}$ that consists of edges corresponding to elements of W is a tree, therefore it contains a vertex $v^i$ of degree 1. It means that when that non-trivial combination is written as a combination of $e^1,e^2,..e^{n+1}$, an element $e^i$ will appear exactly once, therefore that non-trivial combination cannot be equal to 0.

All elements of the image $F(S_G)$ belong to a hyperplane $H:x_1+x_2+...+x_{n+1} = 0$. H is of dimension n, therefore $\{F(e_1),F(e_2),F(e_3),...,F(e_n)\}$ is a basis in H, which in turn means that F is a linear isomorphism between $R^n$ and H, since it converts a basis into a basis.
As a result,vectors $\{d_1,d_2,d_3,...d_l\}\supseteq S_G$ are linear independent, if and only if vectors $\{F(d_1),F(d_2),F(d_3),...F(d_l)\}$ are linear independent.

Let $M_K(S_K, J_K)$ be the matroid of $A_K$ such that $S_K$ is the image $F(S_G)$ .By construction F is an isomorphism between $M_G$ and $M_K$.

Since $M_G$ and $M_K$ are isomorphic, posets $L(M_G)$ and $L(M_K)$ are isomorphic. An isomorphism between $L(M_G)$ and $L(M_K)$ will be denoted as $\hat F$

By definition,
\[\chi_{M_G} (t) = \sum_{x \in L(M_G)} \mu(x)t^{rk(M_G)-rk(x)}\],
\[\chi_{M_K} (t) = \sum_{x \in L(M_K)} \mu(x)t^{rk(M_K)-rk(x)}\].

Let x be an arbitrary element of $L(M_G)$,and let  $\hat F(x) $ be a corresponding element of $L(M_K)$. Since $\mu(x) = \mu(\hat F(x)), rk(x) = rk(\hat F(x)), rk(M_G) = rk(M_K)$, elements of the corresponding sums for x and $\hat F(x)$ in the corresponding formulas are equal, therefore polynomials $\chi_{M_G}$ and $\chi_{M_K}$ are equal.

 Vectors $\{e_1,e_2,...,e_n\}$ and $\{F(e_1),F(e_2),...,F(e_n)\}$ are basis vectors of spaces that contain normal vectors to hyperplanes of  $MA(G,N)$ and $A_K$ respectively, therefore $rk(MA(G,N)) = rk(A_K) = n$. $MA(G,N)$ belongs to $R^n$,$A_K$ belongs to $R^{n+1}$, which means that $\chi_{MA(G,N)} (t) = t^{n-r(MA(G,N))} \cdot \chi_{M_G} (t) = \chi_{M_G} (t), \chi_{A_K} (t) = t^{n+1-r(A_K)} \cdot \chi_{M_K} (t) = t \cdot \chi_{M_K} (t) $, therefore $t \cdot \chi_{A_G} (t) = \chi_{A_K} (t)$. The polynomial $\chi_{A_K} (t)$is equal to the chromatic polynomial of $K_{n+1}$. Since each two vertices of $K_{n+1}$ are adjacent,in proper colouring all vertices are coloured in different colours. With t>n such colouring can be performed in $t(t-1)(t-2)...(t-n)$ ways. This formula is the chromatic polynomial of $K_{n+1}$.Therefore,$\chi_{A_K} (t) = t(t-1)(t-2)...(t-n)$.As a result, $\chi_{MA(G,N)} (t) = (t-1)(t-2)...(t-n) $.
\end{proof}

\begin{lemma}
Let $G_1$ and $G_2$ be two trees with n+1 vertices and n edges. Let $\{e_1, e_2,...,e_n\}$ be a standard basis of n-dimensional space that contains $MA(G_1,N_1)$, such that vectors from this basis correspond to edges of $G_1$. Set $\{e^1,..e^n\}$ is defined for $G_2$ and $MA(G_2,N_2)$ the same way. Let u be a vertex of $G_1$, and let v be a vertex of $G_2$.
The elements of the matroid $M_1$ of the arrangement $MA(G_1,N_1)$ are chosen the following way. If an arrangement has an equation $x_i = 0$, then $e_i$ will be chosen as its normal vector. If an arrangement corresponds to a path that begins in u then the vector $e_{i_1} - e_{i_2} + e_{i_3}-...$, where $e_{i_1}$ corresponds to an edge that is incident to u, will be chosen as its normal vector. For any other hyperplane one of two vectors $e_{i_1} - e_{i_2} + e_{i_3}-...$ will be chosen arbitrarily. 

Then there is a way to choose the elements for the matroid $M_2$ of the arrangement $MA(G_2,N_2)$, such that there is an isomorphism F between $M_1$ and $M_2$, such that if  $a \in M_1$ corresponds to a simple path that begins in u, then F(a) corresponds to a simple path that begins in v.
\end{lemma}

\begin{proof}[Proof]
Let  $c_1,c_2,...,c_n$ be the elements of $M_1$ that correspond to simple paths in $G_1$ that begin in u, and can be written as $e_{i_1} - e_{i_2} + e_{i_3}-...$, where $e_{i_1}$ corresponds to an adjacent to u edge . There are exactly n of these vectors, because for each vertex of $G_1$ that is different from u there is exactly one simple path that connects it to u. 
Let there be a non-trivial linear combination J of vectors $c_1,c_2,...,c_n$ that is equal to 0. Let  $c_{i_1}, c_{i_2}...$ be vectors that have non-zero coefficients in this combination . Let $\hat G_1$ be a subgraph of $G_1$, in which every vertex  and every edge belongs to at least one path that corresponds to one of the vectors $c_{i_1}, c_{i_2}...$. $\hat G_1$ is a tree. Let r be an edge that is incident to a vertex of degree 1 in $\hat G_1$. Among all paths that correspond to vectors $c_{i_1}, c_{i_2}...$ there is only one path that contains this edge. This means that if all vectors $c_{i_1}, c_{i_2}...$ are written as linear combinations of vectors  
$e_1, e_2,...,e_n$, only one of vectors $c_{i_1}, c_{i_2}...$ will have a corresponding to r vector $e_r$ in its combination. Therefore, this vector cannot have a non-zero coefficient in J, because if J is rewritten as a linear combination of $e_1, e_2,...,e_n$, vector $e_r$ will not be reduced. Therefore, there is a contradiction, which means that J does not exist, and vectors $c_1,c_2,...,c_n$ 
are linearly independent.As a result, vectors $c_1,c_2,...,c_n$ form a basis in $R^n$
Let $c_i$  be a vector corresponding to a path between u and $u_i$, and let $c_j$
be a vector corresponding to a path between u and $u_j$, where  $u_i$ and $u_j$ are different from u and from each other. Without loss of generality, let $c_i = e_{i_1} -e_{i_2} + ... +e_{i_k} - e_{j_1} + e_{j_2} -...+ e_{j_m}$, and let $c_j = e_{i_1} -e_{i_2} + ... +e_{i_k} - e_{l_1} + e_{l_2} -...+ e_{l_p}$. Then $c_i - c_j = e_{j_m} - e_{j_{m-1}} +...+e_{j_2} - e_{j_1} + e_{l_1} - e_{l_2} +...-e_{l_p} $,
which is an element of $M_1$ that corresponds to a path between $u_1$ and $u_2$.

As a result, if a path in $G_1$ begins in u, then it corresponds to a vector $c_i$ in $S_1$, otherwise it corresponds to a vector $c_i - c_j$. Let f be a bijective map between vertices of $G_1$ and $G_2$, such that $f(u) = v$. Let $c^1,c^2,...,c^n$ be normal vectors to hyperplanes from $MA(G_2,N_2)$ that correspond to paths in $G_2$ that begin in v, with the following properties:
\begin{enumerate}
    \item 1) $c^i = e^{j_1} - e^{j_2} + e^{j_3} ...$, where $e^{j_1}$ corresponds to an edge that is incident to v.
    \item 2) if $c_i$ corresponds to a path between u and $u_j$, then $c^i$ corresponds to a path between v and $f(u_i)$, where $u_j$ is an arbitrary vertex of $G_1$
    \end{enumerate}
Elements $M_2$ will be chosen the following way. All vectors $c_i$ will be chosen as normal vectors to corresponding hyperplanes. Moreover, if a vector $c_{i_1} - c_{i_2}$ is chosen for a path between  $u_1$ and $u_2$ in $G_1$,then a vector $c^{i_1} - c^{i_2}$ will be chosen for a path between $f(u_1)$ and $f(u_2)$ in $G_2$.

A linear isomorphism F,such that $F(c_i) = c^i$ for every i, defines the desired isomorphism between $M_1$ and $M_2$.
\end{proof}

\begin{remark}
In the next theorem a matrix A will be used. A is a matrix of F for bases $e_1,....,e_n$ and $e^1,....,e^n$.
\end{remark}

\begin{theorem}
(\textit{On tails}): Let $G_1$ and $G_2$ be graphs without loops and parallel edges, such that there are subgraphs $H_1 \subset G_1, T_1 \subset G_1, H_2 \subset G_2, T_2 \subset G_2$ with the following properties:
\begin{enumerate}
    \item $G_1 = H_1 \cup T_1, G_2 = H_2 \cup T_2$
    \item $H_1 \cap T_1 = \{u\}, H_2 \cap T_2 = \{v\}$, where u is a vertex in $G_1$,and v is a vertex in $G_2$
    \item there is an isomorphism F between $H_1$ and $H_2$,that maps u into v
    \item $T_1$ and $T_2$ are trees with the same number of edges, and they both have at least one edge
\end{enumerate}

Then $\chi_{MA(G_1,N_1)} = \chi_{MA(G_2,N_2)}$.
\end{theorem}

\begin{figure}[h]
     \centering
     \includegraphics[width=0.7\textwidth]{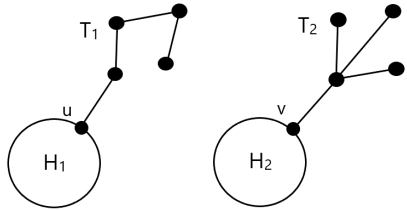}
     \caption{Graphs $G_1$ and $G_2$}
     \label{fig:tail_graph}
 \end{figure}

\begin{proof}[Proof]
  Let $\{e_1,e_2,...e_k,e_{k+1},...e_n\}$ be a basis of space that contains $MA(G_1,N_1)$, such that each $e_i$ corresponds to an edge in $G_1$.Without loss of generality, let $e_1,e_2,...,e_k$ correspond to edges of $H_1$, and let $e_{k+1} ,...,e_n$ correspond to edges of $T_1$. Set $\{e^1,e^2,...e^k, e^{k+1},....e^n\}$ is defined for $G_2$, $H_2$, $T_2$ in the same way , but with condition that for $i<k+1$, if $e_i$ corresponds to an edge r, then $e^i$ corresponds to an edge F(r).
 
Let $M_1(S_1, J_1)$ be the matroid of $MA(G_1,N_1)$. Elements of $M_1$ are chosen the following way. For every hyperplane with an equation $x_i = 0$ vector $e_i$ will be chosen as the normal vector. For every other hyperplane one of two vectors $e_{i_1} - e{i_2} + ...$ is chosen arbitrarily.

 Let A be a matrix of a linear isomorphism from lemma 3 for trees $T_1$ and $T_2$ and vertices u and v. Let U be a linear isomorphism with the following matrix:

\begin{figure}[h]
     \centering
     \includegraphics[width=0.7\textwidth]{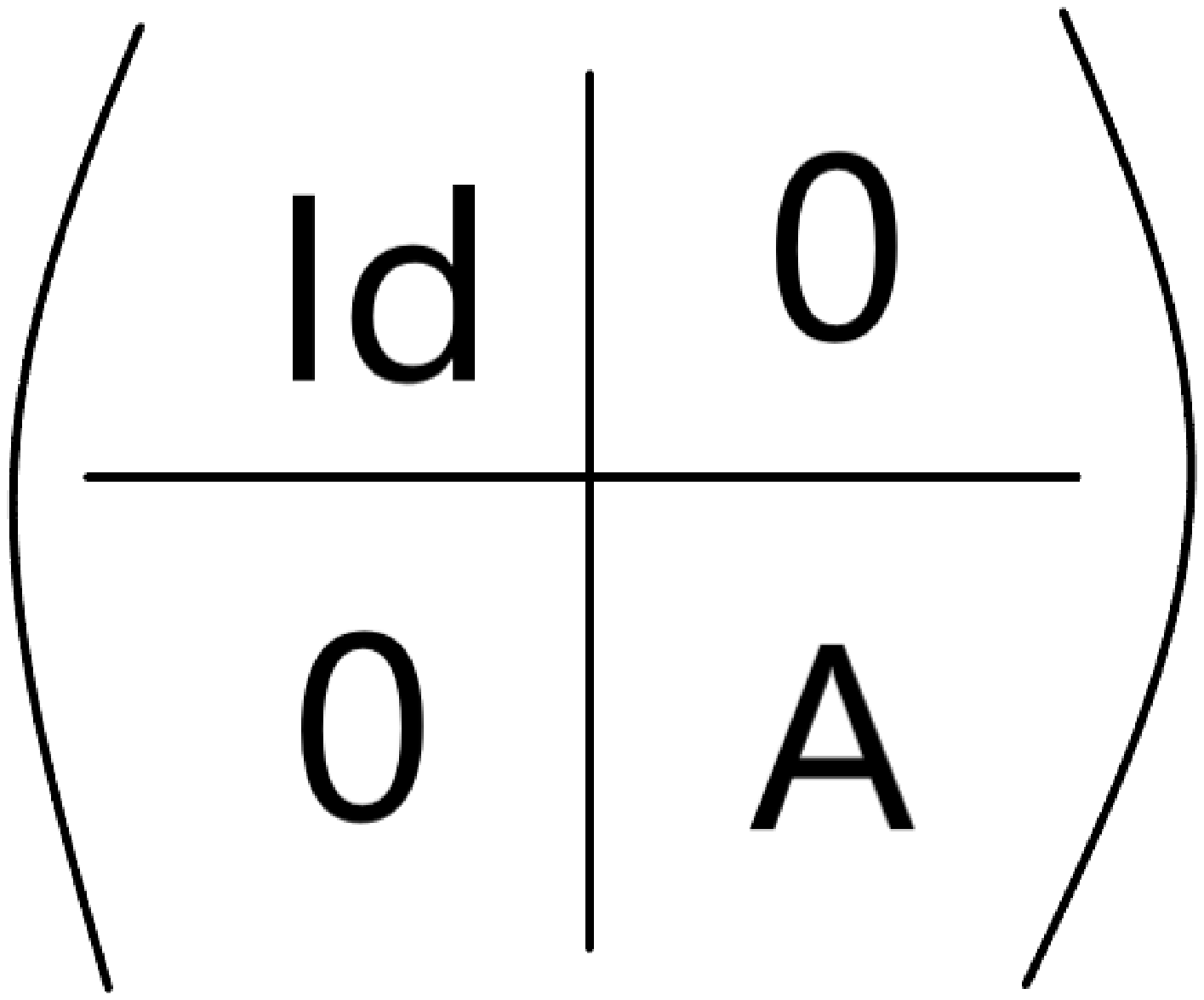}
     \caption{Matrix of U}
     \label{fig:tail_matrix}
 \end{figure}

,where Id  is a matrix of the identity operator, and 0 is a matrix that consists of zeros. Since both Id and A are matrices of linear isomorphisms, U is a linear isomorphism as well.

Let e be an element of $S_1$. If the corresponding to e path consists only of edges that belong to $H_1$ , then by definition of U U(e) is a normal vector to a hyperplane in $MA(G_2,N_2)$ that corresponds to a path in $H_2$. In the same way if all edges of the corresponding to e path lie in $T_1$, then by definition of U U(e) is a normal vector to a hyperplane in $MA(G_2,N_2)$ that corresponds to a path in $T_2$. The last case is when the corresponding path starts in $H_1$ and ends in $T_1$. Without loss of generality, let e be equal to $e_{i_1} + e_{i_2} - .... +e_{i_d} - e_{i_{d+1}}+... +e{i_m}$, where vectors $e_{i_1},...,{e_{i_d}}$ correspond to edges from $H_1$, and vectors $e_{i_{d+1}},...,e_{i_m}$ correspond to edges from $T_1$,with $e_{i_d}$ and $e_{i_{d+1}}$ corresponding to edges that are incident to u. Vector $e_{i_1} + e_{i_2} - .... +e_{i_d}$ corresponds to a simple path in $H_1$, and  $U(e_{i_1} + e_{i_2} - .... +e_{i_d}) = e^{i_1} + e^{i_2} - .... +e^{i_d}$, where $e^{i_d}$ corresponds to an edge in $H_2$ that is incident to v. Vector $e_{i_{d+1}}-... -e{i_m}$ corresponds to a simple path in $T_1$, and   $U(e_{i_{d+1}}-... -e{i_m}) = e^{j_{d+1}}-... e^{j_p}$, where  $e^{j_{d+1}}$ corresponds to and edge in $T_2$ that is incident to v. This means that $U(e_{i_1} + e_{i_2} - .... +e_{i_d} - e_{i_{d+1}}+... +e{i_m}) = e^{i_1} + e^{i_2} - .... +e^{i_d} - e^{j_{d+1}}+... e^{j_p} $, which is a normal vector to a hyperplane in $MA(G_2,N_2)$ that corresponds to a path that begins in $H_2$ and ends in $T_2$.

Let $M_2$ be a matroid of $MA(G_2,N_2)$ whose elements are chosen the following way: $S_2$ is an image of $S_1$  under U. Then U defines a map between $S_1$ and $S_2$. This map is bijective, and since U is a linear isomorphism, this map is an isomorphism of matroids $M_1$ and $M_2$. Since $M_1$ and $M_2$ are isomorphic, posets of matroids $M_1$ and $M_2$ are isomorphic, which in turn means that posets of arrangements $MA(G_1,N_1)$ and $MA(G_2,N_2)$ are isomorphic. Arrangements $MA(G_1,N_1)$ and $MA(G_2,N_2)$ belong to spaces of the same dimension, so $\chi_{MA(G_1,N_1)} = \chi_{MA(G_2,N_2)}$.

\end{proof}

Author is grateful to  A.A.Irmatov for the formulation of the problem and for valuable discussions of these results.

\end{document}